\documentclass[11pt]{amsart}

\usepackage{t1enc}
\usepackage[latin2]{inputenc}
\usepackage{amssymb}
\usepackage{verbatim}
\usepackage{enumerate}
\usepackage{psfrag}
\usepackage{graphicx}

\newtheorem{thm}{Theorem}[section]
\newtheorem{prop}[thm]{Proposition}
\newtheorem{prob}[thm]{Problem}
\newtheorem{dfn}[thm]{Definition}
\newtheorem{clm}[thm]{Claim}
\newtheorem{lemm}[thm]{Lemma}
\newtheorem{cor}[thm]{Corollary}

\newcommand{\lime}{\lim(\omega_1)}
\newcommand{\omg}{\omega_1}
\newcommand{\cont}{2^\omega}
\newcommand{\re}{\mathbb{R}}
\newcommand{\mad}{\mathcal{M}}

\author[D. Soukup]{D\'aniel Soukup}
\address{E\"otv\"os L\'or\'and University}
\email{daniel.t.soukup@gmail.com}
\title{Properties D and \lowercase{a}D are different}
\keywords{linearly D-spaces, aD-spaces, Martin's Axiom, guessing sequences}
\subjclass[2000]{54A25,54A35}

\begin{document}

\maketitle

\begin{abstract}
  Under $(\diamondsuit^*)$ we construct a locally countable, locally compact, 0-dimensional $T_2$ space $X$ of size $\omg$ which is aD however not even linearly D. This consistently answers a question of Arhangel'skii, whether aD implies D. Furthermore we answer two problems concerning characterization of linearly D-spaces, raised by Guo and Junnila.
\end{abstract}

\section{Introduction}

The notion of a \emph{D-space} was probably first introduced by E.K. van Douwen. We recommend G. Gruenhage's paper \cite{gg} which gives a full review on what we know and do not know about D-spaces. A.V.Arhangel'skii and R. Buzyakova defined in \cite{arhadd} a weakening of property D, called \emph{aD}. In \cite{arhcov} Arhangel'skii asked the following:\\

\textbf{Problem 4.6.} Is there a Tychonoff aD-space which is not a D-space?\\

In Section 5 we construct such a space under $(\diamondsuit^*)$. Before that we consider another weakening of property D. Recently H. Guo and H.J.K. Junnila in \cite{linD} introduced the notion of \emph{linearly $D$-spaces} and proved several nice results concerning the topic. In Sections 3 and 4 we answer the following two questions from \cite{linD} in negative (in ZFC):\\

\textbf{Problem 2.5}. Let $X$ be a $T_1$ (linearly) D-space and let $A\subseteq X$ have uncountable regular cardinality. Does $A$ either have a complete accumulation point or a subset of size $|A|$ which is closed and discrete in $X$?\\

\textbf{Problem 2.6.} Is a $T_1$-space $X$ linearly D provided that, for every set $A\subseteq X$ of uncountable regular cardinality, either $A$ has a complete accumulation point or $\overline{A}$ has a subset of size $|A|$ which is closed and discrete in $X$?\\

As we will see these questions rise naturally. The constructions from Sections 3 and 4 can be considered as preparation to the space defined in Section 5. Our construction to Problem 4.6 also answers (consistently) the following from \cite{linD}:\\

\textbf{Problem 2.12.} Is every aD-space linearly D?

\section{Definitions}

An \emph{open neighborhood assignment} (ONA, in short) on a space $(X,\tau)$ is a map $U:X\rightarrow\tau$ such that $x\in U(x)$ for every $x\in X$. X is said to be a \emph{D-space} if for every neighborhood assignment $U$, one can find a closed discrete $D\subseteq X$ such that $X=\bigcup_{d\in D}U(d)=\bigcup U[D]$ (such a set $D$ is called a \emph{kernel for $U$}). In \cite{arhadd} the authors introduced property \emph{aD}:

\begin{dfn}A space $(X,\tau)$ is said to be \emph{aD} iff for each closed $F\subseteq X$ and for each open cover $\mathcal{U}$ of $X$ there is a closed discrete $A\subseteq F$ and $\phi:A\rightarrow \mathcal{U}$ with $a\in\phi(a)$ such that $F\subseteq \cup\phi[A]$.
\end{dfn}

It is clear that D-spaces are aD. A space $X$ is \emph{irreducible} iff every open cover  $\mathcal{U}$ has a \emph{minimal open refinement} $\mathcal{U}_0$; meaning that no proper subfamily of $\mathcal{U}_0$ covers $X$. Later in \cite{arhcov} Arhangel'skii showed the following equivalence.

\begin{thm}[{\cite[Theorem 1.8]{arhcov}}]\label{irreduc}A $T_1$-space $X$ is an aD-space if and only if every closed subspace of $X$ is irreducible.
\end{thm}

Another generalization of property D is due to Guo and Junnila \cite{linD}. For a space $X$ a cover $\mathcal{U}$ is \emph{monotone} iff it is linearly ordered by inclusion.

\begin{dfn}A space $(X,\tau)$ is said to be \emph{linearly D} iff for any ONA $U:X\rightarrow\tau$ for which $\{U(x):x\in X\}$ is monotone, one can find a closed discrete set $D\subseteq X$ such that $X=\bigcup U[D]$.
\end{dfn}

We cite two results from \cite{linD}. A set $D\subseteq X$ is said to be \emph{$\mathcal{U}$-big} for a cover  $\mathcal{U}$ iff there is no $U\in \mathcal{U}$ such that $D\subseteq U$.

\begin{thm}[{\cite[Theorem 2.2]{linD}}] \label{linD} The following are equivalent for a $T_1$-space X:
\begin{enumerate}
  \item X is linearly D.
  \item For every non-trivial monotone open cover $\mathcal{U}$ of $X$, there exists a closed discrete $\mathcal{U}$-big set in $X$.
  \item For every subset $A \subseteq X$ of uncountable regular cardinality $\kappa$, there is a closed discrete subset $B$ of $X$, such that for every neighborhood
$U$ of $B$, we have $|U \cap A| = \kappa$.
\end{enumerate}
\end{thm}

In Problem 2.5 the authors ask whether condition (3) can be made stronger.

\begin{thm}[{\cite[Proposition 2.4]{linD}}] A $T_1$-space $X$ is linearly D if, and only if, for every set $A\subseteq X$ of uncountable regular cardinality, either the set $A$ has
a complete accumulation point or there exists a closed discrete set $D$ of size $|A|$ and a disjoint family $\{A_d: d \in D\}$ of subsets of $A$ such
that $d\in \overline{A_d}$ for every $d\in D$.
\end{thm}

In Problem 2.6 the authors ask whether the second condition of this dichotomy can be weakened.

\section{On Problem 2.5 from \cite{linD}}

In this section we give a negative answer to Problem 2.5. For this let us use the following notion for a space $X$. We say that \emph{$X$ satisfies $(*)$} iff
\begin{itemize}
\item[$(*)$] for every regular, uncountable cardinal $\kappa$ and $A\in[X]^\kappa$ there is a complete accumulation point of $A$ or $A$ has a subset of size $\kappa$ which is closed discrete in $X$.
\end{itemize}
    Problem 2.5 can be rephrased as whether property D implies $(*)$? We will show the following:
\begin{enumerate}
  \item There exists a locally countable $T_2$ D-space $X$ with cardinality $\omg$ which does not satisfy $(*)$.
  \item The existence of a locally countable, locally compact, $T_2$ D-space $X$ with cardinality $<\cont$ which does not satisfy $(*)$ is independent of ZFC.
  \item There exists a locally countable, locally compact $T_3$ (even 0-dimensional) D-space $X$ with cardinality $\cont$ which does not satisfy $(*)$.
\end{enumerate}

First let us observe the following.

\begin{prop}\label{Dunion}Suppose that the space $X$ is the union of a closed discrete set and a D-subspace. Then $X$ is a D-space.
\end{prop}
\begin{proof} Let $X=Y\cup Z$ such that $Y$ is closed discrete, $Z$ is a D-space. Let $U$ be an ONA on $X$. $Z_0=Z\setminus \bigcup\{U(y):y\in Y\}$ is a closed subspace of the D-space $Z$, thus $Z_0$ is a D-space either. There is a closed discrete kernel $D_0$ for the ONA $U|_{Z_0}$ on $Z_0$. Then $D=D_0\cup Y$ is a closed discrete kernel for $U$.
\end{proof}

\begin{prop} There exists a locally countable $T_2$ D-space $X$ with cardinality $\omg$ which does not satisfy $(*)$.
\end{prop}
\begin{proof} Let $X=\omg\times 2$. We define the topology on $X$ as follows. Let $\omg\times \{0\}$ be discrete. For $\alpha<\omg$ let $(\alpha,1)$ have the following neighborhood base:
$$ \bigl\{\{(\alpha,1)\}\cup \bigl((\beta,\alpha)\times\{0\}\bigr):\beta<\alpha\bigr\}.$$
Clearly, $X$ is a locally countable, $T_2$ space. Observe that $\omg\times\{1\}\subseteq X$ is closed discrete, $\omg\times\{0\}\subseteq X$ is discrete, hence $D$. Thus $X$ is a D-space by Proposition \ref{Dunion}. Let $A=\omg\times \{0\}$. Then clearly any infinite subset of $A$ has an accumulation point in $X$. Thus $X$ does not satisfy $(*)$, since there is no full accumulation point of $A$ and any infinite subset of $A$ is not closed discrete in $X$.
\end{proof}

This answers the Problem 2.5 in the negative direction by a $T_2$ counterexample.

\begin{center}
$\star$
\end{center}

The question whether a regular space with this property exists is natural. First we will show, that the existence of a "nice" regular counterexample with cardinality below $\cont$ is independent. We will need a weakening of the \emph{axiom $(t)$} which was introduced by I. Juh\'asz in \cite{juhi}.

\begin{dfn} The \emph{weak (t) axiom}: there exists a \emph{weak (t)-sequences} $\{A_\alpha:\alpha\in \lime\}$, meaning that $A_\alpha \subseteq \alpha$ is an $\omega$-sequence converging to $\alpha$ and for every $X\in[\omega_1]^{\omega_1}$ there is a limit $\alpha$ such that $|X\cap A_\alpha|=\omega$.
\end{dfn}

The existence of such sequences is independent of ZFC. Under MA there is no weak (t)-sequence sequence and adding one Cohen real to any model adds a (weak) (t)-sequence either (see \cite{juhi}).

\begin{prop} Suppose the weak (t)-axiom. Then there exists a locally compact, locally countable, 0-dimensional $T_2$ D-space $X$ which does not satisfy $(*)$.
\end{prop}
\begin{proof} Suppose that $\mathcal{A}=\{A_\alpha:\alpha\in \lime\}$ is a weak (t)-sequence. Let $X=\omg \times 2$. Define the topology on $X$ as follows. Let $\omg\times \{0\}$ be discrete. For $\alpha\in\lime$ let $(\alpha,1)$ have the following neighborhood base: $$\bigl\{\{(\alpha,1)\}\cup \bigl((A_\alpha\setminus\beta) \times \{0\}\bigr): \beta < \alpha\bigr\}.$$ For successor $\alpha<\omg$ let $(\alpha,1)$ be discrete. Clearly, $X$ is a locally countable, locally compact, 0-dimensional $T_2$ space. Notice that $\omg\times\{1\}\subseteq X$ is closed discrete, $\omg\times\{0\}\subseteq X$ is discrete, hence $D$. Thus $X$ is a D-space by Proposition \ref{Dunion}. Let $A=\omg\times \{0\}$. We prove that any uncountable $B\subseteq A$ is not closed discrete in $X$, hence $X$ does not satisfy $(*)$. Let $B_0=\{\alpha<\omg:(\alpha,0)\in B\}\in[\omg]^{\omg}$. Since $\mathcal{A}$ is a weak (t)-sequence, there is $\alpha\in\lime$ such that $|A_\alpha\cap B_0|=\omega$. Clearly, $B$ accumulates to $(\alpha,1)$, thus $B$ is not closed discrete in $X$.
\end{proof}

\textbf{Remark:} In \cite{ishiu} T. Ishiu uses guessing sequences to refine the standard topology on an ordinal.\\

 Now our aim is to prove Proposition \ref{MAkov1} which implies that under MA there is no such space. The following was proved by Z. Balogh  (actually more, but we only need this):

\begin{thm}[{\cite[Theorem 2.2]{balogh}}] \label{balogh}Suppose MA. Then for any locally countable, locally compact space $X$ of cardinality $<2^\omega$ exactly one of the following is true:
\begin{itemize}
  \item X is the countable union of closed discrete subspaces,
  \item X contains a perfect preimage of $\omega_1$ with the order topology.
\end{itemize}
\end{thm}

From this and the following observation we can deduce Proposition \ref{MAkov1}.

\begin{prop}[{\cite[Proposition 7.]{borges}}]\label{borges}If the space $X$ is the countable union of closed D-subspace then $X$ is a D-space.
\end{prop}

\begin{prop} \label{MAkov1}Suppose MA. Then for any locally countable, locally compact space $X$ of cardinality $<2^\omega$ the following are equivalent:
 \begin{enumerate}
 \setcounter{enumi}{-1}
 \item $X$ is $\sigma$-closed discrete,
 \item $X$ is a D-space,
  \item $X$ is a linearly D-space,
  \item $X$ satisfies $(*)$.
  \end{enumerate}
\end{prop}
\begin{proof} The implications $(0)\Rightarrow (1)\Rightarrow (2)$ (by Proposition \ref{borges}) and $(0)\Rightarrow (3)$ are straightforward. $(3)$ implies $(2)$ by Theorem \ref{linD}. We only need to show $(2)\Rightarrow(0)$. Suppose that $X$ is linearly D, by Theorem \ref{balogh} we need to show that $X$ does not contain any perfect preimage of $\omg$.
\begin{clm}\label{ppi}\begin{enumerate}[(i)]
  \item If the space $F$ is a perfect preimage of $\omg$ then $F$ is countably compact, non compact.
  \item If $X$ is first-countable and $F\subseteq X$ is a perfect preimage of $\omega_1$ then $F$ is closed in $X$.
\end{enumerate}
\end{clm}
\begin{proof} (i) It is known that under perfect mappings, the preimage of a compact space is compact (see \cite[Theorem 3.7.2]{eng}). Take any countably infinite $A\subseteq F$ and perfect surjection $f:F\rightarrow \omg$. There is some $\alpha<\omg$ such that $f[A]\subseteq \alpha+1$. Thus $A$ is the subset of the compact set $f^{-1}[\alpha+1]$. (ii) is a consequence of (i).
\end{proof}
Suppose $F\subseteq X$ is a closed subspace. Then $F$ is linearly $D$, hence if $F$ is countably compact $F$ is compact either. By Claim \ref{ppi} $F$ cannot be a perfect preimage of $\omg$.
\end{proof}

\begin{center}
$\star$
\end{center}

Finally we give a regular counterexample to the problem in ZFC without any further set-theoretic assumptions.

\begin{thm} There exists a locally countable, locally compact, 0-dimensional $T_2$ D-space $X$ with cardinality $\cont$ such that $X$ does not satisfy $(*)$.
\end{thm}
\begin{proof}Let $\{C_\alpha: \alpha < \cont\}$ denote an enumeration of the closed dense in itself subsets of $\re$. Let $\{Q_\alpha^\beta:\beta<\cont\}$ denote an enumeration of all countable subsets of $\re$ such that $C_\alpha\subseteq \overline{Q_\alpha^\beta}$ (Euclidean closure taken). Enumerate the pairs $(\alpha,\beta)$ from $\cont \times \cont$ in order type $\cont$: $\{p_\delta:\delta<2^\omega\}$. We define a topology on $X=\re\times 2$ as follows. Let $\re \times \{0\}$ be discrete and we define the topology on $\re \times \{1\}$ by induction. In step $\delta$ for $p_\delta=(\alpha, \beta)$, pick a point $x_\delta\in C_\alpha\setminus \{x_{\delta'}:\delta'<\delta\}$ and let $(x_\delta,1)$ have the following neighborhood base:
$$\bigl\{\{(x_\delta,1)\}\cup\{(x_\delta^n,0):n\geq m\}:m<\omega\bigr\}$$
where $\{x_\delta^n:n<\omega\}\subseteq Q_\alpha^\beta\setminus\{x_\delta\}$ is  any sequence converging to $x_\delta$ in the Euclidean sense. Let the remaining points $(\re\setminus \{x_\delta:\delta<\cont\})\times \{1\}$ be discrete. Clearly, this gives us a locally countable, locally compact, 0-dimensional $T_2$ space. $\re\times\{1\}$ is closed discrete and $\re\times\{0\}$ is discrete, hence a D-space. Thus  $X$ is a D-space by Proposition \ref{Dunion}.

We claim that there is no uncountable subset of $A=\re\times \{0\}\subseteq X$ such that it is closed discrete in $X$ with this topology; this implies that $X$ does not satisfy $(*)$. Let $B\in [A]^{\omg}$ and $B_0=\{x\in\re:(x,0)\in B\}$. Then there is $\alpha<\cont$ such that $C_\alpha\subseteq B_0'$ (where $B_0'$ denotes the Euclidean accumulation points of $B_0$) and $\beta<\cont$ such that $Q_\alpha^\beta\subseteq B$. By definition in step $\delta$ where $p_\delta=(\alpha, \beta)$, we defined the topology on $X$ in such way that $(x_\delta,1)$ is in the closure of $Q_\alpha^\beta\times\{0\}$ thus in the closure of $B$. Thus $B$ is not closed in $X$.
\end{proof}

\section{On Problem 2.6 from \cite{linD}}

Now our aim is to answer Problem 2.6 in negative. For this we will say that a space \emph{$X$ satisfies $(**)$} iff
 \begin{itemize}
 \item[$(**)$]for every regular, uncountable cardinal $\kappa$ and $A\in [X]^{\kappa}$ there is a full accumulation point of $A$ or there is $D\in [\overline{A}]^{\kappa}$ which is closed discrete in $X$.
 \end{itemize}
 Problem 2.6 can be rephrased as whether $(**)$ implies linearly D. We prove the following:
\begin{enumerate}
  \item The existence of a locally countable, locally compact, non linearly D space $X$ with cardinality $<\cont$ which satisfies $(**)$ is independent of ZFC.
  \item There is a locally countable, locally compact, 0-dimensional Hausdorff space $X$ with cardinality $\cont$ which satisfies $(**)$ however not linearly D.
\end{enumerate}

We will use the following notations in this section. $\{A_\alpha:\alpha\in\lime\}$ denotes a \emph{$\clubsuit$-sequence}: for every $\alpha<\omg$ $A_\alpha\subseteq \alpha$ is an $\omega$-type sequence converging to $\alpha$ and for every $X\in[\omg]^{\omg}$ there is some $\alpha<\omg$ such that $A_\alpha\subseteq X$. \emph{Supposing $\clubsuit$} means that there is a $\clubsuit$-sequence. For every $\alpha\in\omg$ enumerate increasingly $A_\alpha$ as $\{a_\alpha^n:n\in\omega\}$. Let $\{M^\beta:\beta\in\omg\}\subseteq [\omega]^\omega$ be an arbitrary almost disjoint family on $\omega$.

\begin{thm} \label{counter} Suppose $\clubsuit$. Then there is a 0-dimensional $T_2$ space $X$ of cardinality $\omg$ such that $X$ is not linearly D, however satisfies $(**)$.
\end{thm}
\begin{proof} First we introduce some further notations for the intervals between the points in the $A_\alpha$'s. For each $\alpha\in\omg$ let $\{I_\alpha^n:n\in\omega\}$ denote the following disjoint open sets in $\omg$:  $I_\alpha^0=(0,a_\alpha^0]$ and $I_\alpha^{n+1}=(a_\alpha^n,a_\alpha^{n+1}]$ for $n\in\omega$. We will define a topology on $X=\omg\times \omg$. Let $\{\alpha\}\times\omg$ be discrete for successor $\alpha$. For $\alpha\in \lime$ and $\beta\in\omg$, a neighborhood base for the point $(\alpha,\beta)$ consists of sets:
$$U((\alpha,\beta),E)=\{(\alpha,\beta)\}\cup\bigcup\{I_\alpha^n\times\omg:n\in M^\beta\setminus E\}$$
where $E\in[\omega]^{<\omega}$. Observe that if $\beta,\beta'\in\omg$ and $\beta\neq\beta'$ then $E=M^\beta\cap M^{\beta'}$ is finite, thus $U((\alpha,\beta),E)\cap U((\alpha,\beta'),E)=\emptyset$.
This way we defined a 0-dimensional $T_2$ topology. Note that the set $X_\alpha=\{\alpha\}\times \omg$ for $\alpha \in \omg$ is closed discrete. Let $\pi(A)=\{\alpha\in\omg:A\cap X_\alpha\neq \emptyset\}$ for $A\subseteq X$.

\psfrag{a}{$\alpha$}
\psfrag{ab}{$(\alpha,\beta)$}
\psfrag{om}{$\omg$}
\psfrag{Xa}{$X_\alpha$}
\psfrag{an}{\tiny{$a_\alpha^n$}}
\psfrag{an1}{\tiny{$a_\alpha^{n+1}$}}
\psfrag{I}{\tiny{$I_\alpha^{n+1}\times\omg$}}
\begin{center}
    \includegraphics[keepaspectratio, width=5 cm]{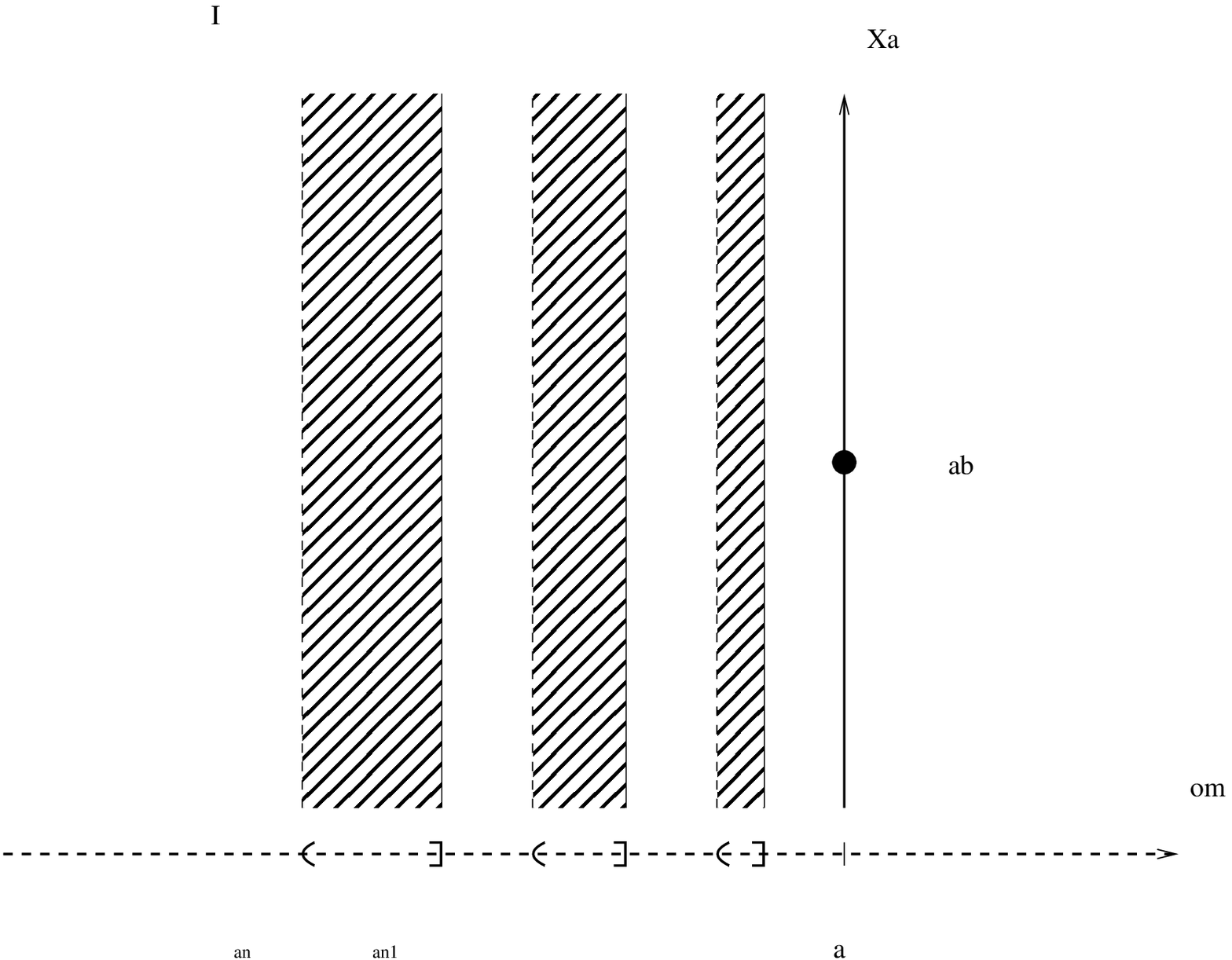}
\end{center}

\begin{clm} \label{treffes}If $|\pi(A)|=\omg$ for $A\subseteq X$ then there are stationary many $\alpha\in\omg$ such that $X_\alpha\subseteq A'$.
\end{clm}
\begin{proof} Since $\{A_\alpha:\alpha\in\lime\}$ is a $\clubsuit$-sequence, the set $S=\{\alpha\in\omg:A_\alpha\subseteq \pi(A)\}$ is stationary. For $\alpha\in S$ we clearly have $X_\alpha\subseteq A'$, since by definition for any $U$ neighborhood of any point $(\alpha,\beta)\in X_\alpha$, $U$ intersects $A$ in infinitely many points.
\end{proof}
This claim has the following corollaries.
\begin{clm} $X$ satisfies $(**)$.
\end{clm}
\begin{proof}Let $A\in [X]^{\omg}$. If there is an $\alpha\in\omg$ such that $|A\cap X_\alpha|=\omg$ we are done. Otherwise for all $\alpha\in\omg$ we have $|A\cap X_\alpha|\leq \omega$ so $|\pi(A)|=\omg$. By Claim \ref{treffes} there is an $\alpha\in\lime$ such that $X_\alpha\subseteq \overline{A}$ and $X_\alpha$ is closed discrete.
\end{proof}
\begin{clm}$X$ is not linearly D.
\end{clm}
\begin{proof}Suppose that $D\subseteq X$ is closed discrete. Then $\pi(D)$ is countable by Claim \ref{treffes}. Hence there is no closed discrete set which is big for the open cover $\{\alpha\times\omg:\alpha<\omg\}$. Thus $X$ is not linearly $D$ by Theorem \ref{linD}.
\end{proof}
This completes the proof of this theorem.
\end{proof}

\textbf{Remark:} If we modify the neighborhoods to be the following for $(\alpha,\beta)$ where $\beta\in\omg$ and $\alpha\in \lime$:
$$\bigl\{\{(\alpha,\beta)\}\cup\bigcup\{I_\alpha^n\times [0,\beta]:n\in M^\beta\setminus E\}:E\in[\omega]^{<\omega}\bigr\}$$
then we obtain a topology which is locally countable, not linearly $D$, satisfies $(**)$ however not even regular.\\

With some further set-theoretic assumptions we can improve the above construction.

\begin{thm}Suppose $\clubsuit$ and CH (equivalently $\diamondsuit$). Then there is a locally countable, locally compact, 0-dimensional $T_2$ space $X$ of cardinality $\omg(=2^\omega)$ such that $X$ is not linearly D, however satisfies $(**)$.
\end{thm}
\begin{proof} For our construction we will need an enumeration of the functions $\omega\rightarrow\omg$ as $\{F_\delta:\delta<\omg\}$; here we used CH. Let $h$ be an arbitrary bijection $h:\omg\rightarrow\omg\times\omg$. We define a topology on $X=\omg\times\omg$. Let $X_\alpha=\{\alpha\}\times \omg$, $X_{<\alpha}=\alpha\times \omg$. Define neighborhoods for points in $X_\alpha$ by induction on $\alpha$. Let $(X_{<\alpha},\tau_{<\alpha})$ denote the topology defined by the induction till step $\alpha$. We have the following conditions which we will preserve during each step:
\begin{enumerate}[(i)]
  \item $(X_{<\alpha},\tau_{<\alpha})$ is locally countable, locally compact, 0-dimensional,
   \item $X_{<\beta}$ is open in $X_{<\alpha}$ for $\beta<\alpha$,
   \item for every $\beta<\alpha$ and  $(\beta,\gamma)\in X_{\beta}$ there is some neighborhood $G$ of  $(\beta,\gamma)$ such that $G\setminus\{(\beta,\gamma)\}\subseteq X_{<\beta}$,
  \item for every $\beta'<\beta<\alpha$ the set $(\beta',\beta]\times\omg$ is clopen in $(X_{<\alpha},\tau_{<\alpha})$.
\end{enumerate}
This way we will get a topology $\tau$ on $X$ by taking $\cup\{\tau_{<\alpha}:\alpha<\omg\}$ as a base.

For successor $\alpha<\omg$ just let $(\alpha,\beta)$ be a discrete point for any $\beta<\omg$. Clearly this way the inductional hypothesis will hold. Suppose now that $\alpha$ is limit, we define a neighborhood for $(\alpha,\beta)$ (for any $\beta<\omg$) as follows. Suppose that $h(\beta)=(\delta, \gamma)$. For each $n\in\omega$ take a countable, compact, clopen $G_n\subseteq (a_\alpha^{n-1},a_\alpha^n]\times\omg$ such that $(a_\alpha^n,F_\delta(n))\in G_n$; this can be done by (i) and (iv). Then define the neighborhoods of $(\alpha,\beta)$ by the following base:
$$\bigl\{\{(\alpha,\beta)\}\cup\bigcup\{G_n:n\in M^\beta\setminus E\}:E\in[\omega]^{<\omega}\bigr\}.$$
It is clear that the inductional assumptions will hold for the resulting topology. It follows from the construction that $(X,\tau)$ is locally countable, locally compact and 0-dimensional. Thus we constructed a space $(X,\tau)$ which refines the topology from the previous theorem, hence $X$ is $T_2$ either.  Let $\pi(A)=\{\alpha\in\omg:A\cap X_\alpha\neq \emptyset\}$ for $A\subseteq X$.
\begin{clm} \label{treffes2}If $|\pi(A)|=\omg$ for $A\subseteq X$ then there are stationary many $\alpha\in\omg$ such that $|X_\alpha\cap A'|=\omg$.
\end{clm}
\begin{proof} Since $\{A_\alpha:\alpha\in\lime\}$ is a $\clubsuit$-sequence, the set $S=\{\alpha\in\omg:A_\alpha\subseteq \pi(A)\}$ is stationary. Fix an $\alpha\in S$. Define $F:\omega\rightarrow \omg$ such that $(a_\alpha^n,F(n))\in A$ where $A_\alpha=\{a_\alpha^n:n\in\omega\}$. So there is some $\delta$ for which we have $F_\delta=F$. We claim that $\{(\alpha,\beta)\in X_\alpha:\exists\gamma<\omg:h(\beta)=(\delta,\gamma)\}\subseteq A'$. Clearly for such an $(\alpha,\beta)$ we used $F_\delta=F$ in the induction to define the neighborhoods, from which we see that the set $\{(a_\alpha^n,F(n)):n\in\omega\}\subseteq A$ accumulates to $(\alpha,\beta)$.
\end{proof}
\begin{clm} $X$ satisfies $(**)$.
\end{clm}
\begin{proof}Let $A\in[X]^{\omg}$. If there is an $\alpha\in\omg$ such that $|A\cap X_\alpha|=\omg$ we are done. Otherwise since $|A\cap X_\alpha|\leq \omega$ we have $|\pi(A)|=\omg$. By Claim \ref{treffes2} $\overline{A}$ intersects (stationary) many closed discrete sets $X_\alpha$,  in $\omg$ many points.
\end{proof}
\begin{clm}$X$ is not linearly D.
\end{clm}
\begin{proof}Suppose that $D\subseteq X$ is closed discrete. Then $\pi(D)$ is countable by Claim \ref{treffes2}. Hence there is no closed discrete set which is big for the open cover $\{\alpha\times\omg:\alpha<\omg\}$. Thus $X$ is not linearly $D$ by Theorem \ref{linD}.
\end{proof}
This completes the proof of this theorem.
\end{proof}

We can further extend the equivalences of Proposition \ref{MAkov1} using Theorem \ref{balogh}.

\begin{prop} Suppose MA. Suppose that the space $X$ is locally countable, locally compact of cardinality less than $\cont$. Then the following are equivalent:
\begin{enumerate}
  \setcounter{enumi}{0}
  \item $X$ is (linearly) D,
  \setcounter{enumi}{3}
  \item $X$ satisfies $(**)$.
\end{enumerate}
\end{prop}
\begin{proof} (1) implies (4) trivially. Suppose (4), then a closed uncountable subspace of $X$ cannot be countably compact. Then by Claim \ref{ppi} there is no perfect preimage of $\omg$ in $X$. By Theorem \ref{balogh} this implies that $X$ is $\sigma$-closed-discrete, hence (linearly) D.
\end{proof}

\begin{center}
$\star$
\end{center}

Under ZFC, without any further set-theoretic assumptions we can give a counterexample.

\begin{thm} There is a locally countable, locally compact, 0-dimensional $T_2$ space $X$ such that $X$ is not linearly D however satisfies $(**)$.
\end{thm}
\begin{proof} We will use the following notations: let $\{C_\alpha:\alpha<\cont\}$ be an enumeration of uncountable closed dense in itself subsets of $\mathbb{R}$ and enumerate $\{Q\in[\mathbb{R}\setminus \mathbb{Q}]^\omega:C_\alpha\subseteq \overline{Q}\}$ as $\{Q_\alpha^\beta:\beta<\cont\}$. Enumerate the triples $(C_\alpha,Q_\alpha^\beta,\gamma)$ for $\alpha,\beta,\gamma<\cont$ in order type $\cont$: $\{t_\delta:\delta<\cont\}$. We need an enumeration of all functions $F:\omega\rightarrow\cont$, $\{F_\varphi:\varphi<\cont\}$. Fix an $h:\cont\rightarrow\cont\times\cont$ bijection. Furthermore, let $\{M^\epsilon:\epsilon<\cont\}\subseteq [\omega]^\omega$ be an almost disjoint family on $\omega$.

 We define a topology on $X=\cont\times\cont$ by induction.  Let $X_\delta=\{\delta\}\times \cont$ for $\delta<\cont$. Let $(X_{<\delta},\tau_{<\delta})$ denote the topology defined by the induction till step $\delta<\cont$ where $X_{<\delta}=\bigcup \{X_{\delta'}:\delta'<\delta\}$. In step $\delta$ we pick a point $x_\delta$ from the real line which will help us define the neighborhoods of points in $X_\delta$. We have the following conditions which we preserve during the induction:
\begin{enumerate}[(i)]
  \item $(X_{<\delta},\tau_{<\delta})$ is locally countable, locally compact, 0-dimensional,
  \item $X_{<\delta'}$ is open in $X_{<\delta}$  for $\delta'<\delta$,
  \item for every $\delta'<\delta$ and  $(\delta',\epsilon)\in X_{\delta'}$ there is some neighborhood $G$ of $(\delta',\epsilon)$ such that $G\setminus\{(\delta',\epsilon)\}\subseteq X_{<\delta'}$,
  \item \emph{property (E)}: suppose $\delta'<\delta$ and $x_{\delta'}\in B$ where $B$ is Euclidean open. If $(\delta',\epsilon)\in X_{\delta'}$ then there is some compact, countable and clopen neighborhood $G$ of $(\delta',\epsilon)$ such that $G\subseteq \bigcup\{X_{\delta''}:x_{\delta''}\in B\}$.
\end{enumerate}
This way we will get a topology $\tau$ on $X$ if we take $\cup\{\tau_{<\delta}:\delta<\cont\}$ as a base.

Suppose we are in step $\delta\in \cont$, where $t_\delta=(C_\alpha,Q_\alpha^\beta,\gamma)$. We do the following:
\begin{itemize}
  \item pick a point $x_\delta\in C_\alpha\setminus(\{x_{\delta'}:\delta'<\delta\}\cup \mathbb{Q})$,
  \item if the set $Q_\alpha^\beta\cap\{x_{\delta'}:\delta'<\delta\}$ does not accumulate to $x_\delta$ just let each point of $X_\delta$ be discrete,
  \item if the set $Q_\alpha^\beta\cap\{x_{\delta'}:\delta'<\delta\}$ accumulates to $x_\delta$, choose a sequence $\{x_{\delta'_n}:n\in\omega\}\subseteq Q_\alpha^\beta\cap\{x_{\delta'}:\delta'<\delta\}$ converging to $x_\delta$,
  \item take disjoint open intervals $B_n$ with rational endpoints, containing $x_{\delta'_n}$.
\end{itemize}
Now we are ready to define a neighborhood of a point $(\delta,\epsilon)$. Suppose $h(\epsilon)=(\varphi,\rho)$.
\begin{itemize}
  \item Consider the points $(\delta'_n,F_\varphi(n))$ in $X_{\delta'_n}$,
  \item by property (E) we can take compact, countable and clopen neighborhoods $G_n$ of $(\delta'_n,F_\varphi(n))$ such that $G_n\subseteq \bigcup\{X_{\delta''}:x_{\delta''}\in B_n\}$. Observe that $\bigcup\{G_n:n\in\omega\}$ is closed in $(X_{<\delta},\tau_{<\delta})$.
\end{itemize}
Let
$$U((\delta,\epsilon),E)=\{(\delta,\epsilon)\}\cup\bigcup\{G_n:n\in M^\epsilon\setminus E\}$$
for $E\in[\omega]^{<\omega}$ and let the following be a neighborhood base for $(\delta,\epsilon)$
$$\bigl\{U((\delta,\epsilon),E):E\in[\omega]^{<\omega}\bigr\}.$$
Note that if $\epsilon\neq\epsilon'<\cont$ and $E=M^\epsilon\cap M^{\epsilon'}$ then $U((\delta,\epsilon),E)\cap U((\delta,\epsilon'),E)=\emptyset$; this yields that the resulting topology will be $T_2$. We need to check that the inductional assumptions still hold. Clearly $U((\delta,\epsilon),E)$ is countable and compact, we need to check that it is clopen. Since $U((\delta,\epsilon),E)\cap X_{<\delta}=\bigcup\{G_n:n\in\omega\}$ is closed in $X_{<\delta}$, we only need to check that $(\delta,\epsilon')\notin \overline{U((\delta,\epsilon),E)}$ for $\epsilon\neq\epsilon'<\cont$. Let $F=M^\epsilon\cap M^{\epsilon'}\in[\omega]^{<\omega}$ then $U((\delta,\epsilon),E)\cap U((\delta,\epsilon'),F)=\emptyset$. Properties (ii) and (iii) will clearly hold. We need to check (iv), property (E). For points in $X_{<\delta}$ this will still hold. Consider a Euclidean open $B$ such that $x_\delta\in B$ and a new point: $(\delta,\epsilon)\in X_\delta$. Using the notations of the definition of a basic neighborhood for $(\delta,\epsilon)$, there is some $m\in\omega$ such that $\bigcup\{B_n:n\geq m\}\subseteq B$. So for the following neighborhood:
$$G=\{(\delta,\epsilon)\}\cup\bigcup\{G_n:n\in M^\epsilon, n\geq m\}$$
we have that $G\subseteq \bigcup\{X_{\delta''}:x_{\delta''}\in B\}$, since for $n\geq m$ $B_n\subseteq B$ thus $G_n\subseteq \bigcup\{X_{\delta''}:x_{\delta''}\in B_n\} \subseteq \bigcup\{X_{\delta''}:x_{\delta''}\in B\}$.

 It is clear that $(X,\tau)$ is a locally countable, locally compact and 0-dimensional space. It is straightforward by property (iii) that each $X_\delta$ is closed discrete. Let $\pi(A)=\{\delta<\cont:A\cap X_\delta\neq\emptyset\}$ and $\pi_0(A)=\{x_\delta:\delta\in\pi(A)\}\subseteq \mathbb{R}$ for $A\subseteq X$.
\begin{clm} \label{treffes3} If $|\pi(A)|>\omega$ for $A\subseteq X$ then there are $\cont$ many $\delta<\cont$ such that $|X_\delta \cap A'|=\cont$.
\end{clm}
\begin{proof} There is $\alpha<\cont$ such that $C_\alpha\subseteq \overline{\pi_0(A)}$ (Euclidean closure taken) and $\beta<\cont$ such that $Q_\alpha^\beta\subseteq \pi_0(A)$. Let
\begin{eqnarray*}
D=\bigl\{\delta<\cont:\exists\gamma<\cont\bigl(t_\delta=(C_\alpha,Q_\alpha^\beta,\gamma)\bigr) \\ \text{ and }\forall\delta'<\cont(x_{\delta'}\in Q_\alpha^\beta\Rightarrow \delta'<\delta)\bigr\}.
 \end{eqnarray*}
 Take a $\delta\in D$. Clearly we did not defined $X_\delta$ to be discrete since all points in $Q_\alpha^\beta$ are of the form $x_{\delta'}$ where $\delta'<\delta$. So at step $\delta$ in the induction we chose some convergent sequence $\{x_{\delta'_n}:n\in\omega\}$ from $Q_\alpha^\beta$ where $\delta'_n<\delta$. Let $F:\omega\rightarrow \cont$ such that $(\delta_n',F(n))\in A$. There is some $\varphi\in\cont$ such that $F=F_\varphi$. We claim that $\{(\delta,\epsilon)\in X_\delta:\exists \rho<\cont:h(\epsilon)=(\varphi,\rho)\}\subseteq A'$. For such a point $(\delta,\epsilon)$, we used $F_\varphi=F$ for the definition of basic neighborhoods, thus the set $\{(\delta'_n,F(n)):n\in\omega\}\subseteq A$ accumulates to $(\delta,\epsilon)$.
\end{proof}
\begin{clm} $X$ satisfies $(**)$.
\end{clm}
\begin{proof}Let $A\in[X]^{\kappa}$ such that $\kappa$ is an uncountable, regular cardinal. If there is a $\delta<\cont$ such that $|A\cap X_\delta|=\kappa$ we are done. Otherwise $|\pi(A)|>\omega$ since $|A\cap X_\delta|< \kappa$ for all $\delta<\cont$. By Claim \ref{treffes3} $\overline{A}$ intersects (continuum) many closed discrete sets $X_\delta$, in $\cont$ many points.
\end{proof}
\begin{clm}$X$ is not linearly D.
\end{clm}
\begin{proof}Suppose that $D\subseteq X$ is closed discrete. Then $\pi(D)$ is countable by Claim \ref{treffes3}. Hence there is no closed discrete set which is big for the open cover $\{X_{<\delta}:\delta<\cont\}$. Thus $X$ is not linearly $D$ by Theorem \ref{linD}.
\end{proof}
This completes the proof of this theorem.
\end{proof}

\textbf{Remark:} P. Nyikos gave an example of a space T which is not a D-space, however for every  closed $F\subseteq T$: $e(F)=L(F)$. From \cite[Theorem 1.11]{nyikos} from his article, one can see that $T$ is linearly D (use the characterization of linear D property by Theorem \ref{linD}). Applying Claim \ref{treffes3} to our last construction we get the following:

\begin{cor} There exists a Hausdorff space $X$ of cardinality $\cont$ such that $X$ is locally countable, locally compact, 0-dimensional, not linearly D however $e(F)=L(F)$ for every closed subset $F\subseteq X$.
\end{cor}

\section{Consistently on property D and \lowercase{a}D}

Our main goal in this section is to construct a space which is not linearly D, however every closed subset of it is irreducible; hence aD by Theorem \ref{irreduc}.

We will use the following set-theoretical assumption:
\begin{enumerate}[$(\diamondsuit^*)$]
\item there is a \emph{$\diamondsuit^*$-sequence}, meaning that there exists an $\{\mathcal{A}_\alpha:\alpha\in\lime\}$ such that $\mathcal{A}_\alpha\subseteq [\alpha]^\omega$ is countable and for every $X\subseteq \omg$ there is a club $C\subseteq \omg$ such that $X\cap \alpha \in \mathcal{A}_\alpha$ for all $\alpha\in C$.
\end{enumerate}
Before proving the theorem we need the following easy claim about maximal almost disjoint families (MAD, in short).

\begin{clm}\label{specmad}If $\{N_i:i\in\omega\}\subseteq [\omega]^\omega$ then there is a MAD family $\mad\subseteq [\omega]^\omega$ of size $\cont$ such that for all $M\in \mad$ and $i\in\omega$: $|M\cap N_i|=\omega$.
\end{clm}
\begin{proof} We will construct the MAD family $\mad$ on $\mathbb{Q}$. We can suppose that each $N_i$ is dense in $\mathbb{Q}$.  Let $\mathbb{R}=\{x_\alpha:\alpha<\cont\}$ and for all $\alpha<\cont$ let $S_\alpha\subseteq \mathbb{Q}$ such that $S_\alpha$ is a convergent sequence with limit point $x_\alpha$ and $|S_\alpha\cap N_i|=\omega$ for all $i\in\omega$. Then $\mathcal{S}=\{S_\alpha:\alpha<\cont\}$ is almost disjoint, let $\mathcal{T}=\{T_\alpha:\alpha<\lambda\}\subseteq [\mathbb{Q}]^\omega$ such that $\mathcal{S}\cup\mathcal{T}$ is MAD. Then $\mad=\{S_\alpha\cup T_\alpha:\alpha<\lambda\}\cup\{S_\alpha:\lambda\leq\alpha<\cont\}$ is a MAD family with the desired property.
\end{proof}

\begin{thm} Suppose $(\diamondsuit^*)$. There is a locally countable, locally compact, 0-dimensional $T_2$ space $X$ of size $\omg$ such that X is not linearly D, however every closed subset $F\subseteq X$ is irreducible; equivalently $X$ is an aD-space.
\end{thm}
\begin{proof} We will define a topology on $X=\omg\times\omg$. Let $X_\alpha=\{\alpha\}\times\omg$ and $X_{<\alpha}=\alpha\times \omg$ for $\alpha<\omg$.
 \begin{dfn}The set $A\in [X]^\omega$ \emph{runs up to $\alpha<\omg$} iff $A=\{(\alpha_n,\beta_n):n\in\omega)\}\subseteq X_{<\alpha}$ such that $\alpha_0\leq...\leq\alpha_n\leq...$ and $\sup\{\alpha_n:n\in\omega\}=\alpha$.
 \end{dfn}
 Note that if $A\subseteq X$ runs up to some $\alpha<\omg$ then $A\cap X_\beta$ is finite for all $\beta<\omg$.

  We need the following consequence of $(\diamondsuit^*)$.  Let $\pi(A)=\{\alpha\in\omg:A\cap X_\alpha\neq \emptyset\}$ for $A\subseteq X$.
\begin{clm} \label{seq} $(\diamondsuit^*)$  There exists a sequence $\{A_\alpha:\alpha\in\lime\}\subseteq [X]^\omega$ with $A_\alpha=\bigcup\{A_\alpha^n:n\in\omega\}$  for all $\alpha\in\lime$ such that
\begin{enumerate}
\item $|A_\alpha^n|=\omega$ for all $n\in\omega$,
\item $A_\alpha$ runs up to $\alpha$,
\item for all $Y\subseteq X$ if $|\pi(Y)|=\omg$ then $$\exists \text{ club } C\subseteq\omg \text{ such that } \forall\alpha\in C\exists n\in\omega (A_\alpha^n\subseteq Y).$$
\end{enumerate}
\end{clm}
\begin{proof} Let $\{\mathcal{A}_\alpha:\alpha\in\lime\}$ denote a $\diamondsuit^*$-sequence. Let $i:\omg\times\omg\rightarrow\omg$ denote a bijection which maps $\bigl((\alpha+1)\times(\alpha+1)\bigr)\setminus(\alpha\times\alpha)$ to $\omega\cdot(\alpha+1)\setminus\omega\cdot\alpha$. Let $$\widetilde{\mathcal{A}}_\alpha=\{i^{-1}(A):A\in\mathcal{A}_{\omega\cdot\alpha},\sup\bigl(\pi(i^{-1}(A))\bigr)=\alpha\}$$
and let $A_\alpha=\bigcup\{A_\alpha^n:n\in\omega\}$ such that
\begin{enumerate}
\item $|A_\alpha^n|=\omega$ for all $n\in\omega$,
\item $A_\alpha$ runs up to $\alpha$,
\item[$(3)'$] for all $B\in\widetilde{\mathcal{A}}_\alpha$ there is $n\in\omega$ such that $A_\alpha^n\subseteq B$,
\end{enumerate}
for all $\alpha\in\lime$. We claim that the sequence $\{A_\alpha:\alpha\in\lime\}$ has the desired properties. Let $Y\subseteq X$ such that $|\pi(Y)|=\omg$. There is some club $C_0\subseteq \omg$ such that $Y\cap X_{<\alpha}\subseteq \alpha\times\alpha$ for $\alpha\in C_0$. There is some club $C_1\subseteq\omg$ such that $\alpha\cap i[Y]\in \mathcal{A}_\alpha$ for $\alpha\in C_1$. Let $C_2=\{\alpha<\omg:\omega\cdot\alpha\in C_1\}$; clearly, $C_2$ is a club. Let $C=C_0\cap C_2\cap\pi(Y)'$. Fix some $\alpha\in C$. Then $\omega\cdot\alpha\cap i[Y]=A$ for some $A\in\mathcal{A}_{\omega\cdot\alpha}$, thus $i[Y\cap X_{<\alpha}]=A$ since $\omega\cdot\alpha=i[\alpha\times\alpha]$ and $Y\cap X_{<\alpha}\subseteq \alpha\times\alpha$. Hence $i^{-1}(A)=Y\cap X_{<\alpha}$ and $i^{-1}(A)\in\widetilde{\mathcal{A}}_\alpha$ because $\alpha\in\pi(Y)'$. Thus there is $n\in\omega$ such that $A_\alpha^n\subseteq Y$ by $(3)'$.
\end{proof}
Let $\{A_\alpha:\alpha\in\lime\}\subseteq [X]^\omega$ denote a sequence with $A_\alpha=\bigcup\{A_\alpha^n:n\in\omega\}$ for $\alpha\in\lime$ from Claim \ref{seq}. We want to define the topology on $X$ such that
\begin{itemize}
          \item $X_\alpha$ is closed discrete for all $\alpha<\omg$,
          \item $X_{<\alpha}$ is open for all $\alpha\in\omg$,
          \item if $A\in [X]^\omega$ runs up to $\alpha$ then $A$ has an accumulation point in $X_\alpha$,
          \item $X_\alpha\subseteq \overline{A_\alpha^n}$ for all $\alpha\in\lime$ and $n\in\omega$.
 \end{itemize}
Let $\mad_\alpha\subseteq [A_\alpha]^\omega$ denote a MAD family on $A_\alpha$ for $\alpha\in\lime$ such that $|M\cap A_\alpha^n|=\omega$ for all $M\in\mad_\alpha$ and $n\in\omega$; such an $\mad_\alpha$ exists by Claim \ref{specmad}. Enumerate $\mad_\alpha=\{M_\alpha^\beta:\beta<\omg\}$.

 We define topologies $\tau_{<\alpha}$ on $X_{<\alpha}$ by induction on $\alpha<\omg$ such that $\tau_{<\alpha}\cap\mathcal{P}({X_{<\beta}})=\tau_{<\beta}$ for all $\beta<\alpha<\omg$. This way we will get a topology $\tau$ on $X$ if we take $\cup\{\tau_{<\alpha}:\alpha<\omg\}$ as a base.

  Suppose $\alpha<\omg$ and we have defined the topology $(X_{<\alpha},\tau_{<\alpha})$ such that
\begin{enumerate}[(i)]
\item $(X_{<\alpha},\tau_{<\alpha})$ is a locally countable, locally compact, 0-dimension-al $T_2$ space,
\item for all $\alpha'<\alpha$ and $x\in X_{\alpha'}$ there is some neighborhood $G$ of $x$ such that $G\cap X_{\alpha'}=\{x\}$,
    \item $(\alpha_0,\alpha_1]\times\omg\subseteq X_{<\alpha}$ is clopen for all $\alpha_0<\alpha_1<\alpha$.
\end{enumerate}
 If $\alpha\in\omg\setminus\lime$ then let $X_\alpha$ be discrete. Suppose $\alpha\in\lime$ and let us enumerate $\{F\subseteq X_{<\alpha}\setminus A_\alpha:F \text{ runs up to }\alpha\}$ as $\{F_\alpha^\beta:\beta<\omg\}$.

 \begin{dfn}A subspace $A\subseteq T$ of a topological space $T$ is \emph{completely discrete} iff there is a discrete family of open sets $\{G_a:a\in A\}$ such that $a\in G_a$ for all $a\in A$.
 \end{dfn}

 The following claim will be useful later.

 \begin{clm} \label{compld}Suppose that $A=\{(\alpha_n,\beta_n):n\in\omega\}\subseteq X$ runs up to $\alpha$. Then $A$ is completely discrete in $X_{<\alpha}$; hence closed discrete either.
 \end{clm}
 \begin{proof} Let $G_0=(0,\alpha_0]\times\omg$ and $G_{n+1}=(\alpha_n,\alpha_{n+1}]\times\omg$ for $n\in\omega$. $G_n$ is open for all $n\in\omega$ by inductional hypothesis (iii). Note that $\{G_n:n\in\omega\}$ is a discrete family of open sets such that $A\cap G_n$ is finite for all $n\in\omega$. Let $\mathcal{G}_n$ denote a finite, disjoint family of clopen subsets of $G_n$ such that for all $a\in A\cap G_n$ there is exactly one $G\in\mathcal{G}_n$ such that $a\in G$. Then the discrete family $\cup\{\mathcal{G}_n:n\in\omega\}$ shows that $A$ is completely discrete.
 \end{proof}

 In step $\alpha\in\lime$ we define the neighborhoods of points in $X_\alpha=\{(\alpha,\beta):\beta<\omg\}$ by induction on $\beta<\omg$ such that:
\begin{enumerate}[$(a)$]
\item $X_{<\alpha}\cup\{(\alpha,\beta'):\beta'\leq\beta\}$ is locally countable, locally compact and 0-dimensional $T_2$,
    \item there is some neighborhood $U$ of $(\alpha,\beta)$ such
     that $U\cap A_\alpha\subseteq M_\alpha^\beta$,
\item $M_\alpha^\beta$ converges to $(\alpha,\beta)$,
\item $F_\alpha^\beta$ accumulates to $(\alpha,\beta')$ for some $\beta'\leq\beta$.
\end{enumerate}
We need the following lemma to carry out the induction on $\beta<\omg$.

\begin{lemm}\label{indukc} Suppose that $(T\cup S,\tau)$ is a locally countable, locally compact and 0-dimensional $T_2$ space such that $T$ is open and $S$ is countable. Let $D=\{d_n:n\in\omega\}\subseteq T$ closed discrete in $T\cup S$ and completely discrete in $T$. Let $r\notin T\cup S$. Then there is a topology $\rho$ on $R=T\cup S\cup \{r\}$ such that
\begin{itemize}
\item $(R,\rho)$ is locally countable, locally compact and 0-dimensional $T_2$,
\item $\rho|_{(T\cup S)}=\tau$,
\item $D$ converges to $r$ and $r\notin \overline{S}$ in $(R,\rho)$.
\end{itemize}
\end{lemm}
\begin{proof} Suppose that $d_n\in G_n$ such that $\{G_n:n\in\omega\}$ is a family of open sets which is  discrete in $T$. For each $n\in\omega$ let $\{B_i^n:i\in\omega\}$ denote a neighborhood base of $d_n$ such that
\begin{itemize}
\item $G_n\supseteq B^n_0\supseteq B^n_1\supseteq ...$ and
\item $B^n_i$ is countable, compact and clopen for all $n,i\in\omega$.
\end{itemize}
Since $S\cap D=\emptyset$ there is some clopen neighborhood $U_s$ of each $s\in S$ such that $U_s\cap D=\emptyset$. There is $g_s:\omega\rightarrow\omega$ such that $$U_s\cap B^n_{g_s(n)}=\emptyset \text{ for all } n\in\omega.$$
Since $S$ is countable, there is $g:\omega\rightarrow\omega$ such that for all $s\in S$ there is some $N\in\omega$ such that $g_s(n)\leq g(n)$ for all $n\geq N$. Define the topology $\rho$ on $R$ as follows. Let $$B_N=\{r\}\cup\bigcup\{B^n_{g_(n)}:n\geq N\}\text{ and }\mathcal{B}=\{B_N:N\in\omega\}.$$
Let $\rho$ be the topology on $R$ generated by $\tau\cup \mathcal{B}$.

 Clearly $\rho|_{(T\cup S)}=\tau$. We claim that $(R,\rho)$ is locally countable, locally compact and 0-dimensional. Since $\mathcal{B}$ is a neighborhood base for $r$, it suffices to prove that each $B\in\mathcal{B}$ is countable, compact (trivial) and clopen. Let $N\in\omega$ then $B_N$ is clopen in $T$ since $\bigcup\{B^n_{g(n)}:n\in\omega\}$ is a family of clopen sets which is discrete in $T$ guaranteed by the discrete family $\{G_n:n\in\omega\}$. Let $s\in S$. There is $N\in\omega$ such that $U_s\cap B^n_{g(n)}=\emptyset$ for $n\geq N$. There is some neighborhood $V\in\tau$ of $s$ such that $V\cap\bigcup\{B^n_{g(n)}:n<N\}=\emptyset$ since $s$ is not in the closed set $\bigcup\{B^n_{g(n)}:n<N\}$. Thus $(U_s\cap V)\cap B_N=\emptyset$. This proves that $B_N$ is clopen.

  We claim that $(R,\rho)$ is $T_2$. Let $s\in S$, then there is $N\in\omega$ such that $U_s\cap B^n_{g(n)}=\emptyset$ for $n\geq N$, thus $B_N\cap U_s=\emptyset$. As noted before $B_N\cap T$ is closed and clearly $\bigcap\{B_N\cap T:N\in\omega\}=\emptyset$. This yields that any point $t\in T$ and $r$ can be separated, thus $(R,\rho)$ is $T_2$.

  Clearly $D$ converges to $r$ and $S\cap B=\emptyset$ for any $B\in\mathcal{B}$ thus $r\notin \overline{S}$.
\end{proof}

Suppose we are in step $\beta<\omg$ and we defined the neighborhoods of points in $X_{<\alpha}\cup\{(\alpha,\beta'):\beta'<\beta\}$. We use Lemma \ref{indukc} to define the neighborhoods of $r=(\alpha,\beta)$. Let $T=X_{<\alpha}$ and $S=\{(\alpha,\beta'):\beta'<\beta\}\cup(A_\alpha\setminus M_\alpha^\beta)$. Note that $F_\alpha^\beta\cup M_\alpha^\beta$ runs up to $\alpha$ thus closed and completely discrete in $T$ by Claim \ref{compld}. Also, $M_\alpha^\beta$ is closed discrete in $T\cup S$ by inductional hypothesis (b) for $(\alpha,\beta')$ where $\beta'<\beta$.
\begin{itemize}
\item If $F_\alpha^\beta$ accumulates to $x_{\beta'}$ for some  $\beta'<\beta$ then let $D=M_\alpha^\beta$.
    \item If $F_\alpha^\beta$ is closed discrete in $T\cup S$ then let $D=M_\alpha^\beta\cup F_\alpha^\beta$.
\end{itemize}
 Note that $D$ is closed discrete in $T\cup S$. By Claim \ref{indukc} we can define the neighborhoods of $r=(\alpha,\beta)$ such that the resulting space satisfies conditions (a), (b),(c) and (d). After carrying out the induction on $\beta$, the resulting topology on $X_\alpha$ clearly satisfies conditions (i),(ii) and (iii). This completes the induction.

  As a base, the family $\bigcup\{\tau_{<\alpha}:\alpha\in\lime\}$ generates a topology $\tau$ on $X$ which is locally countable, locally compact and 0-dimensional $T_2$. Observe that $X_\alpha$ is closed discrete and $X_{<\alpha}$ is open for all $\alpha<\omg$ (by inductional hypothesises (ii) and (iii)) .

\begin{clm}\label{sorkomp2}Suppose that $F\subseteq X$ runs up to some $\alpha\in\lime$. Then there is some $\beta<\omg$ such that $F$ accumulates to $(\alpha,\beta)$. Equivalently, if $G\subseteq X$ is open and $X_\alpha\subseteq G$ then there is some $\alpha'<\alpha$ such that $(\alpha',\alpha]\times\omg\subseteq G$.
\end{clm}
\begin{proof} There is some $\beta<\omg$ such that $F=F_\alpha^\beta$. Thus by inductional hypothesis (d) there is some $\beta'\leq\beta$ such that $F$ accumulates to $(\alpha,\beta')$.
\end{proof}
\begin{clm} $X$ is not linearly D.
\end{clm}
\begin{proof} If $D\subseteq X$ is closed discrete then $\pi(D)$ is finite by Claim \ref{sorkomp2}. Thus there is no big closed discrete set for the cover $\{X_{<\alpha}:\alpha<\omg\}$.
\end{proof}
Our next aim is to prove that all closed subspaces of $X$ are irreducible.
\begin{clm}\label{dclaim2} If $|\pi(F)|=\omega$ for a closed $F\subseteq X$ then $F$ is a D-space, hence irreducible.
\end{clm}
\begin{proof} Since $F=\cup\{F\cap X_\alpha:\alpha\in\pi(F)\}$ is a countable union of closed discrete sets, $F$ is a D-space by Proposition \ref{borges}. We mention that if the ONA $U$ on $F$ has closed discrete kernel $D$ then we get an irreducible cover by taking the following open refinement: $\{(U(d)\setminus D)\cup \{d\}:d\in D\}$.
\end{proof}
\begin{clm}\label{club} If $|\pi(A)|=\omg$ for $A\subseteq X$ then there is a club $C\subseteq \omg$ such that $C\times \omg\subseteq A'$. As a consequence, if $\pi(U)$ is stationary for the open $U\subseteq X$ then there is some $\alpha<\omg$ such that $X\setminus U\subseteq \alpha\times \omg$.
\end{clm}
\begin{proof} There is a club $C\subseteq \omg$ by Claim \ref{seq} such that for all $\alpha\in C$ there is $n\in\omega$ such that $A_\alpha^n\subseteq A$. We will prove that $X_\alpha\subseteq A'$ for $\alpha\in C$. Take any point $(\alpha,\beta)\in X_\alpha$. $|M_\alpha^\beta\cap A_\alpha^n|=\omega$ for all $\beta<\omg$ by the construction of the MAD family $\mad_\alpha$ and $M_\alpha^\beta$ converges to $(\alpha,\beta)$ by inductional hypothesis (c). Thus $A_\alpha^n$ accumulates to $(\alpha,\beta)$, hence $X_\alpha\subseteq A'$.
\end{proof}
\begin{clm} If $|\pi(F)|=\omg$ for a closed $F\subseteq X$ then $F$ is irreducible.
\end{clm}
\begin{proof}Take an open cover of $F$, say $\mathcal{U}$. We can suppose that we refined it to the form $\mathcal{U}=\{U(x):x\in F\}$, where $U(x)$ is a neighborhood of $x\in F$. From Claim \ref{club} we know that there is some club $C\subseteq \omg$ such that $C\times \omg \subseteq F$. For $\alpha\in C$ define the open set $G_\alpha=\cup\{U(x):x\in X_\alpha\}$. For every $\alpha\in C$ there is some $\delta(\alpha)<\alpha$ such that $(\delta(\alpha),\alpha]\times \omg\subseteq G_\alpha$; by Claim \ref{sorkomp2}. So there is some $\delta<\omg$ and a stationary $S\subseteq C$ such that $(\delta,\alpha]\times \omg\subseteq G_\alpha$ for all $\alpha \in S$. Fix some $\delta_0>\delta$ such that $X_{\delta_0}\subseteq F$. Let $S_0=S\setminus(\delta_0+1)$. For all $\alpha \in S_0$ there is $d_\alpha\in X_\alpha\subseteq F$ such that $(\delta_0,\alpha)\in U(d_\alpha)$. Let us refine these sets: $U_0(d_\alpha)=\bigl(U(d_\alpha)\setminus (\{\delta_0\}\times S_0)\bigr)\cup\{(\delta_0,\alpha)\}$ for all $\alpha\in S_0$; let $\mathcal{U}_0=\{U_0(d_\alpha):\alpha \in S_0\}$. Clearly $\mathcal{U}_0$ is an open refinement of $\mathcal{U}$ which is minimal and $\{d_\alpha:\alpha\in\omg\}\subseteq\cup\mathcal{U}_0$. Since $S_0$ is stationary and $S_0\subseteq \pi[\cup\mathcal{U}_0]$ we get that there is some $\gamma<\omg$ such that $F_1=F\setminus \cup\mathcal{U}_0\subseteq \gamma\times\omg$ by Claim \ref{club}. So by Claim \ref{dclaim2} the closed set $F_1$ is a D-space, hence irreducible. Take a minimal open refinement of the cover $\{U(x)\setminus (\{\delta_0\}\times S_0):x\in F_1\}$, let this be $\mathcal{U}_1$. The union $\mathcal{U}_0\cup \mathcal{U}_1$ is an open refinement of $\mathcal{U}$ which covers $F$ and minimal.
\end{proof}
This proves that all closed subspaces of $X$ are irreducible. Hence $X$ is an aD-space by Theorem \ref{irreduc}.

\end{proof}

Using again the strong result of Balogh we can observe the following.

\begin{prop} Suppose MA. Let $X$ be a locally countable, locally compact space of cardinality $<\cont$. Then the following are equivalent:
\begin{enumerate}
\item $X$ is a (linearly) D-space,
\setcounter{enumi}{4}
\item $X$ is an aD-space.
\end{enumerate}
\end{prop}
\begin{proof} (1) implies (5) trivially. Suppose that $X$ is an aD-space. It is enough to show that $X$ does not contain any perfect preimage of $\omg$. Since property aD is hereditary to closed sets, any closed countably compact subspace is compact. By Claim \ref{ppi}, there is no perfect preimage of $\omg$ in $X$.
\end{proof}

\begin{cor} The existence of a locally countable, locally compact space $X$ of size $\omg$ which is aD and non linearly D is independent of ZFC.
\end{cor}

However, the following remains opens.

\begin{prob}\begin{enumerate}
              \item Is it consistent with ZFC that there exists a locally countable, locally compact space $X$ of cardinality $<\cont$ such that $X$ is not (linearly) D however aD?
               \item Is there a ZFC example of a Tychonoff space $X$ such that $X$ is not (linearly) D however aD?
               \item[(2)$'$] Is there a ZFC example of a locally countable, locally compact (0-dimensional) $T_2$ space $X$ such that $X$ is not (linearly) D however aD?
            \end{enumerate}
\end{prob}


\begin{thebibliography}{9}

\bibitem{arhadd} A.V.Arhangel'skii and R. Buzyakova, Addition theorems and D-spaces, Comment. Mat. Univ.
Car. 43(2002), 653-663.

\smallskip

\bibitem{arhcov} A.V. Arhangel'skii, D-spaces and covering properties, Topology and Appl. 146-147(2005), 437-
449.

\smallskip

\bibitem{balogh} Z. Balogh, Locally nice spaces and Martin's Axiom, Comment. Math. Univ. Carolin. 24 (1983)
    63-87.

\smallskip

\bibitem{borges} C.R. Borges and A. Wehrly, A study of D-spaces, Topology Proc. 16(1991), 7-15.

\smallskip

\bibitem{eng} R. Engelking, General Topology, Berlin: Heldermann, 1989.

\smallskip


\bibitem{gg} G. Gruenhage, A survey of D-spaces, to appear. (http://www.auburn.edu/$\sim$gruengf/papers/dsurv5.pdf)

\smallskip

\bibitem{linD} H. Guo and H.J.K. Junnila, On spaces which are linearly D, Topology and Appl., Volume 157, Issue 1, 1 January 2010, Pages 102-107.

\smallskip

\bibitem{ishiu} T. Ishiu, A non-D-space with large extent, Topology Appl. 155 (2008), 1256-1263.

\smallskip

\bibitem{juhi} I. Juh\'asz, A weakening of $\clubsuit$, with applications to topology, Comment. Math. Univ. Carolin. 29, 4 (1988).

\smallskip

\bibitem{nyikos} P. Nyikos, D-spaces, trees, and an answer to a problem of Buzyakova, to appear. (http://www.math.sc.edu/$\sim$nyikos/Dspaces.pdf)
 \end{thebibliography}
\end{document}